\documentclass[12pt]{article}
\usepackage{amsmath,amssymb,amstext,dsfont,fancyvrb,float,fontenc,graphicx,subfigure, theorem}

\usepackage[letterpaper]{geometry}
\ifx\volno\undefined\def\volno{1}\fi
\ifx\volyear\undefined\def\volyear{2017}\fi
\ifx\papno\undefined\def\papno{P0.0}\fi

\newfont{\footsc}{cmcsc10 at 8truept}
\newfont{\footbf}{cmbx10 at 8truept}
\newfont{\footrm}{cmr10 at 10truept}

\usepackage{fancyhdr}
\usepackage{relsize}
\usepackage{sectsty}
\allsectionsfont{\larger[-1]} 

\renewcommand\paragraph{\@startsection{paragraph}{4}{\z@}
                                    {2ex \@plus.5ex \@minus.2ex}
                                    {-1em}
                                    {\normalfont\normalsize\bfseries}}

\renewcommand\subparagraph{\@startsection{subparagraph}{5}{\parindent}
                                       {2ex \@plus.5ex \@minus .2ex}
                                       {-1em}
                                      {\normalfont\normalsize\bfseries}}

\newlength{\BiblioSpacing}
\renewenvironment{thebibliography}[1]{
\begin{oldthebibliography}{#1}
\setlength{\parskip}{\BiblioSpacing}
\setlength{\itemsep}{\BiblioSpacing}
}
{
\end{oldthebibliography}
}

\usepackage[strict]{changepage}
\def\abstractname{Abstract -}   % <-----------------
\def\abstract{\begin{adjustwidth}{1cm}{1cm} \par    \footnotesize \noindent {\bf \abstractname} 
\def\endabstract{ \end{adjustwidth} \smallskip }}
{\theorembodyfont{\itshape}\newtheorem{theorem}{Theorem}[section]}
{\theorembodyfont{\itshape}\newtheorem{proposition}[theorem]{Proposition}}
{\theorembodyfont{\itshape}\newtheorem{definition}[theorem]{Definition}}
{\theorembodyfont{\itshape}}
{\theorembodyfont{\itshape}\newtheorem{corollary}[theorem]{Corollary}}
{\theorembodyfont{\rm}}
{\theorembodyfont{\rm}}
{\theorembodyfont{\rm}}
{\theorembodyfont{\rm }}

\def\dedicatory{\date}

\title{\Large\bf Differences of Harmonic Numbers and the $abc$-Conjecture}
\author{\sc N. da Silva, S. Raianu, and H. Salgado\footnote{This work was supported by a PUMP Undergraduate Research Grant (NSF Award DMS-1247679)}}
\dedicatory{\normalsize\em Dedicated to the memory of Lauren\c{t}iu Panaitopol (1940-2008)}

\begin{document}
\maketitle
\thispagestyle{fancy}

\vskip 1.5em

\begin{abstract}
Our main source of inspiration was a talk by Hendrik Lenstra on harmonic numbers, which are numbers whose only prime factors are two or three. Gersonides proved 675 years ago that one can be written as a difference of harmonic numbers in only four ways: 2-1, 3-2, 4-3, and 9-8.  We investigate which numbers other than one can or cannot be written as a difference of harmonic numbers and we look at their connection to the $abc$-conjecture. We find that there are only eleven numbers less than 100 that cannot be written as a difference of harmonic numbers (we call these $ndh$-numbers). The smallest $ndh$-number is 41, which is also Euler's largest lucky number and is a very interesting number. We then show there are infinitely many $ndh$-numbers, some of which are the primes congruent to $41$ modulo $48$. For each Fermat or Mersenne prime we either prove that it is an $ndh$-number or find all ways it can be written as a difference of harmonic numbers. Finally, as suggested by Lenstra in his talk, we interpret Gersonides' theorem as ``The $abc$-conjecture is true on the set of harmonic numbers" and we expand the set on which the $abc$-conjecture is true by adding to the set of harmonic numbers the following sets (one at a time): a finite set of $ndh$-numbers, the infinite set of primes of the form $48k+41$, the set of Fermat primes, and the set of Mersenne primes.
\end{abstract}
 

\begin{MSC}
11A07; 11A41; 11D45
\end{MSC}

\section{Preliminary results} 
Lenstra's talk \cite{hl} starts with the following definition introduced by the bishop, music theorist, poet, and composer Philippe de Vitry, a.k.a. Philippus De Vitriaco (1291-1361):
\begin{definition}\label{harmnum}
A {\em\bf harmonic number} is a number that can be written as a power of two times a power of three.
\end{definition}
Vitry found the following consecutive pairs of harmonic numbers: 1,2; 2,3; 3,4; 8,9. These pairs correspond to the frequency ratios in the following musical intervals: octave, perfect fifth, perfect fourth, major second (or whole tone). (In music, intervals with frequency ratios a power of two over a power of three, or vice versa, are called Pythagorean intervals.) He asked whether these are the only pairs of consecutive harmonic numbers, and his question was answered in the affirmative by the mathematician, philosopher, astronomer, and Talmudic scholar Levi ben Gershom, a.k.a. Gersonides (1288-1344). In his talk Lenstra gives the details of the original proof of Gersonides, whose idea was to look at remainders modulo 8. We will give a different proof, using methods that are similar to some that we will use in the other sections of this paper (in sections 2 and 3 we will also make abundant use of remainders modulo 8).
\begin{theorem}\label{gersonides}(Gersonides, 1342) The only two consecutive harmonic numbers greater than $4$ are $8$ and $9$.
\end{theorem}
\begin{proof}
If two harmonic numbers are consecutive, then one of them is a power of two and the other one is a power of three. We assume first that  $3^n=2^m+1$ and that $m>1$, so also $n>1$. Then we have $(2+1)^n=2^m+1$, and using the binomial theorem we obtain:
$$2^n+n2^{n-1}+\ldots+\frac{n(n-1)}{2}2^2+n2+1=2^m+1,$$
so after subtracting 1 from both sides and dividing by 2 we get
$$2^{n-1}+n2^{n-2}+\ldots+n(n-1)+n=2^{m-1}.$$
Since $n(n-1)$ is even, we get that $n=2k$ for some integer $k$, and therefore $3^{2k}=2^m+1$. We now look at the last digit of the number on the left: if $k=2l$ this last digit is 1, which contradicts the fact that no power of 2 ends in 0. So $k=2l+1$, and thus $3^{4l+2}=2^m+1$, or $(3^{2l+1}-1)(3^{2l+1}+1)=2^m$. In conclusion, $3^{2l+1}-1=2^s$ and if $l\ne 0$, then, as above, we obtain that $2l+1$ is even, a contradiction. Thus $l=0$, so $n=2$ and $m=3$.\\
The other case is easier: we assume that $3^n=2^m-1$ and $n>1$, so $m>2$. Then, again we have
$(2+1)^n=2^m-1$, and using the binomial theorem we obtain:
$$2^n+n2^{n-1}+\ldots+\frac{n(n-1)}{2}2^2+n2+1=2^m-1,$$
so after adding 1 to both sides and dividing by 2, we get
$$2^{n-1}+n2^{n-2}+\ldots+n(n-1)+n+1=2^{m-1}.$$
So $n$ is odd. But we can also write $(4-1)^n=2^m-1$, so 
$$4^n-n4^{n-1}+\ldots+n4-1=2^m-1.$$
After adding 1 to both sides and dividing by 4 we get that $n$ is even, a contradiction.
\end{proof}
At the end of his talk, Lenstra mentions the famous $abc$-conjecture (see \cite{gt}). Roughly speaking, it states that if the coprime (i.e. with no common prime factors) positive integers $a$, $b$, and $c$ satisfy $a+b=c$, and if we denote by $rad(n)$ the product of all prime divisors of $n$, then usually $rad(abc)$ is not much smaller than $c$. One version of the precise statement is the following:\\[3mm]
{\bf The $abc$-conjecture} (Oesterl\'{e}-Masser) For any $\varepsilon>0$ there exist only finitely many triples $(a,b,c)$ of coprime positive integers for which $a+b=c$ and $c>rad(abc)^{1+\varepsilon}$.\\[3mm]
Lenstra explains in his talk that the equalities involved in Theorem \ref{gersonides}, namely $3^n=2^m+1$ and $3^n+1=2^m$ roughly look like solutions to the equation from Fermat's Last Theorem (i.e. $x^n+y^n=z^n$, just allow the exponents to be different), and that it is known that Fermat's Last Theorem can be derived from the $abc$-conjecture (see \cite{gt}).  The direct connection between Theorem \ref{gersonides} and the $abc$-conjecture is the following:
\begin{corollary}\label{abcharm}
The $abc$-conjecture is true on the set of harmonic numbers.
\end{corollary}
\begin{proof}
The only way to pick three coprime numbers from the set of harmonic numbers is the following: one of the numbers has to be 1, another one is a power of two, and the last one is a power of three. Therefore, the corollary follows directly from Gersonides' Theorem.
\end{proof}
Gersonides' Theorem is also connected to another famous conjecture, proposed by  Catalan in 1844 and solved in 2002 by Preda Mih\u{a}ilescu. It is now called Mih\u{a}ilescu's Theorem:
\begin{theorem} \cite{m} The only integer solutions greater than or equal to 2 of the equation
$$x^z-y^t=1$$
are $x=3, y=2, z=2, t=3$.
\end{theorem}
Mih\u{a}ilescu's proof uses cyclotomic fields and Galois modules, but a weaker version of his result, \cite[Theorem 2, p. 146]{gp1}, which assumes that $x$ and $y$ are prime, can be proved with elementary techniques similar to the ones used in our proof of Theorem \ref{gersonides}.  If we change the definition of harmonic numbers by replacing 3 with another odd prime, the first thing we notice is that we lose the music applications and therefore the justification for the name. Other than that, \cite[Theorem 2]{gp1} becomes the analog of Theorem \ref{gersonides}: it just says that there will be no consecutive ``new  harmonic" numbers. Corollary \ref{abcharm} will also remain true but it would be less interesting, mainly because the solution $1+2^3=3^2$ has small radical: in this case $rad(abc)=rad(2^33^2)=6<c=9$ making it a ``high quality" solution.\\[2mm]

One of our goals will be to expand the set on which the $abc$-conjecture is true by adding other numbers to the set of harmonic numbers. Even adding just one single number can be tricky, e.g. proving that the $abc$-conjecture holds on the set of harmonic numbers and the number $5$ is quite hard (see the proof of Theorem \ref{fermatdiff} ii)).

\begin{figure}[!htb]
\begin{center}
% use \includegraphics to import figures 
\includegraphics[width=154mm,scale=0.8]{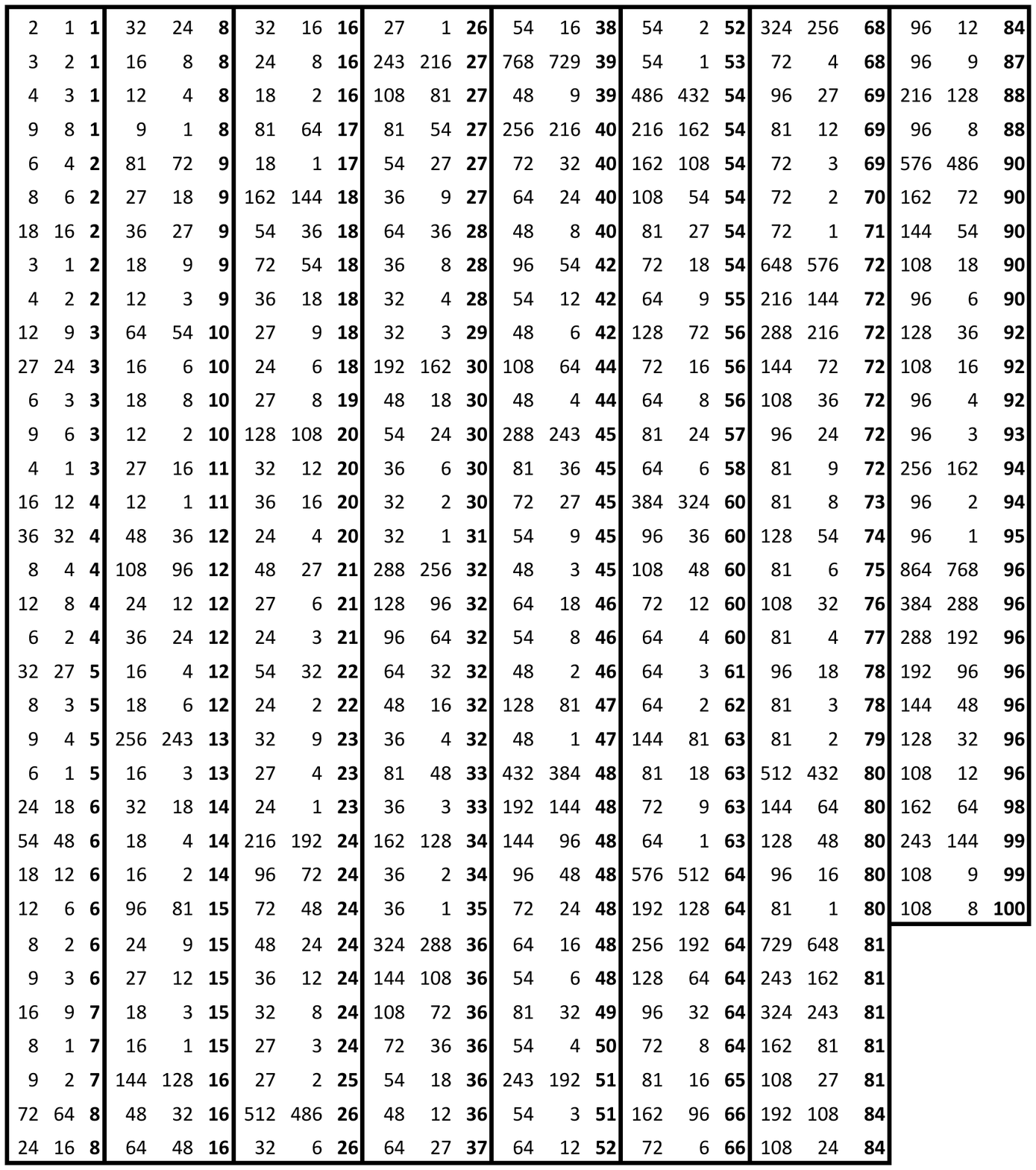}
\end{center}
\caption{\label{fig1} Differences (up to 100) of harmonic numbers less than 1000.}
\end{figure}

\section{Numbers that cannot be written as differences of harmonic numbers}
We inspected the table of harmonic numbers less than 1000 given below:
$$\begin{array}{ccccccc}
1&3&9&27	&81&243&729\\
2&6&18&54&162&486&\\
4&12	&36&108&324&972&\\
8&24&72&216&648&&\\
16&48&144&432&&&\\
32&96&288&864&&&\\
64&192&576&&&&\\
128&384&&&&&\\
256&768&&&&&\\
512&&&&&&
\end{array}$$
and we saw that the first few tens of natural numbers can all be written as a difference of harmonic numbers in this table. Then we asked whether there are any positive integers that cannot be written as a difference of harmonic numbers.

\begin{definition}
A positive integer is called an {\em\bf $ndh$-number} if it cannot be written as a difference of harmonic numbers.
\end{definition}

In Figure \ref{fig1} we have an Excel table listing all one and two digit differences of harmonic numbers in the above table ordered from 1 to 100. We noticed that there are eleven numbers missing, and we checked with a Java program that these eleven numbers cannot be written as differences of harmonic numbers with higher exponents. In the next result we prove that these eleven numbers are  $ndh$-numbers, so together with Figure  \ref{fig1} this shows that these integers are the only $ndh$-numbers in the first 100 positive integers.
\begin{theorem}\label{ndhn100}
The numbers $41, 43, 59, 67, 82, 83, 85, 86, 89, 91,$ and $97$ are $ndh$-numbers.
\end{theorem}
\begin{proof}
Among the eleven numbers we have nine odd and two even.\\
 We focus on the odd ones first, and note that none of them are divisible by 3.
If one of them is a difference of harmonic numbers, that difference is either $2^m-3^n$ or $3^n-2^m$.\\
We show first that none of them can be written as $2^m-3^n$, where $m\ge 6$. We start with 85 and we see that if $85=2^m-3^n$, remainders modulo 8 tell us that $n$ must be odd. Then $3^n$ ends in 3 or 7, so $2^m$ ends in 8 or 2, respectively. This means that $m$ is also odd. Then we have
$(3-1)^m=3^n+85,$
or 
$3^m-m3^{m-1}+\ldots+3m-1=3^n+85,$
and this is a contradiction because $3\nmid 86$. So 85 cannot be written as $2^m-3^n$.\\
All the remaining eight odd numbers are either of the form $8k+1$ (41, 89, 97) or $8k+3$. None of them can be written as $2^m-3^n$, which has remainder modulo 8 either 7 or 5. In conclusion, none of the nine odd numbers in the statement can be written as $2^m-3^n$.\\
Now we show that none of the odd numbers can be written as $3^n-2^m$. Again we start with 85 and we see that if $85=3^n-2^m$, then $m\ge 8$ and the remainders modulo 8 are 5 on the left and 1 or 3 on the right, a contradiction. So 85 cannot be written as $3^n-2^m$.\\
We next show that none of the numbers of the form $8k+1$, i.e. 41, 89, and 97 can be written as $3^n-2^m$. Since it is clear that $m\ge 4$, we have that the remainder of $3^n$ modulo 8 is 1, so $n=2s$ is even. Since 41 and 89 have remainder 2 modulo 3, in their cases $m=2t$ is also even. Also if $97=3^n-2^m=(8+1)^s-2^m=8^s+s8^{s-1}+\ldots+8s+1-2^m$, so $s$ is even, which means that $3^n$ ends in 1, hence $2^m$ ends in 4, thus $m=2t$ is even for all three numbers. But then $3^n-2^m=(3^s-2^t)(3^s+2^t)$, so $3^s-2^t=1$ because all of 41, 89 and 97 are prime. By Theorem \ref{gersonides} we get that either $s=2$ and $t=3$ or $s=t=1$, but none of them is possible.\\
We now show that none of the remaining odd numbers, which are all of the form $8k+3$, can be written as $3^n-2^m$. Taking remainders modulo 8 we see that for all of them $n=2s+1$ has to be odd. Then $3^n$ ends in 3 or 7. It follows that $3^n-43$ ends in 4 and $3^n-67$ ends in 6 (because no power of 2 ends in 0). This means that for both 43 and 67 $m$ would be even and we would have
$$(3-1)^m=3^m-m3^{m-1}+\ldots-3m+1=3^n-l,$$
where $l$ is either 43 or 67. This cannot happen because $3\nmid 44\cdot 68$. This shows that none of 43 and 67 can be written as a difference of harmonic numbers.\\
We now show that none of 59, 83, and 91 can be written as $3^{2s+1}-2^m$. Indeed, if
\begin{equation}\label{last3}
3(8+1)^s=2^m+u,
\end{equation}
where $u\in\{59,83,91\}$, then
$$ 3\cdot 8^s+3s8^{s-1}+\ldots+3\cdot 8s+3=2^m+u,$$
so
$$ 3\cdot 8^s+3s8^{s-1}+\ldots+3\cdot 8s=2^m+v,$$
where $v\in\{56,80,88\}$, and after dividing by 8 we get
 $$ 3\cdot 8^{s-1}+3s8^{s-2}+\ldots+3s=2^{m-3}+w,$$
where $w\in\{7,10,11\}$.\\
For $w=7$ we get that $3s-7$ is even, so $s$ is odd, and hence the left hand side of (\ref{last3}) ends in 7. Then $2^m$ ends in 8 and $m$ is odd. Then $3^n=3^m-m3^{m-1}+\ldots+3m-1+59$, which is a contradiction because $3\nmid 58$.\\
For $w=10$ we get that $3s-10$ is even, so $s$ is even,  hence the left hand side of (\ref{last3}) ends in 3. Then $2^m$ ends in 0, a contradiction.\\
For $w=11$ we get that $s$ is odd, and hence the left hand side of (\ref{last3}) ends in 7. Then $2^m$ ends in 6 and $m=2t$ is even. Then $3^n=4^t+91=3^t+3^{t-1}+\ldots+3t+1+91$, which is a contradiction because $3\nmid 92$. In conclusion, none of the nine odd numbers in the statement can be written as a difference of harmonic numbers.\\
We end by showing that neither 82 nor 86 can be written as a difference of harmonic numbers. They cannot be written as a difference of even harmonic numbers because 41 and 43 are $ndh$-numbers. Then they would have to be written as a difference of two odd harmonic numbers, i.e. $3^n-1$. This is not possible because neither 83 nor 87 are powers of 3, and the proof is complete. 
\end{proof}

The smallest $ndh$-number (i.e. 41) appears in a lot of places playing many roles, like a character actor. It is Euler's largest lucky number, it is also a Newman-Shanks-Williams prime, a Sophie Germain prime,
an Eisenstein prime, a Proth prime, and (according to the theologian and musicologist Friedrich Smend) it even appears in the works of Bach (yes, the composer!). Smend claimed in \cite{sm} that J. S. Bach regularly used the natural-order alphabet (which assigns numbers to letters: A=1, I,J=9, U,V=20, Z=24, and then assigns to each word the sum of the numbers corresponding to the letters in that word). One of Smend's examples is (see \cite{t}) the {\it Canon a 4 voce} written in 1713 for his second cousin Johann Gottfried Walther, in which Smend claims that Bach used his own last name as the number of bars:
 $$\begin{array}{ccccccccc}
B&&A&&C&&H&=&14\\
2&+&1&+&3&+&8&&
\end{array}$$
and his cousin's last name as the number of sounding notes:
 $$\begin{array}{ccccccccccccccc}
W&&A&&L&&T&&H&&E&&R&=&82.\\
21&+&1&+&11&+&19&+&8&+&5&+&17&&
\end{array}$$

Smend also points out that Bach's full name is exactly half of Walther's last name:
$$\begin{array}{ccccccccccccc}
J.&&S.&&B&&A&&C&&H&=&41.\\
9&+&18&+&2&+&1&+&3&+&8&&
\end{array}$$

Smend's theory was adopted by many people who interpreted the number of bars and notes in Bach's scores according to the natural-order alphabet. Musicologist Ruth Tatlow studied the plausibility of Smend's claims in \cite{t}, challenged his conclusions, and recommended caution in using his theory. As far as our paper is concerned, the last two numbers (41 and 82) are $ndh$-numbers, while 14 can be written as a difference of harmonic numbers in the following ways: $16-2, 18-4,$ and $32-18$ (see Figure \ref{fig1} and Theorem \ref{mersennediff} ii) for the proof). As we will soon see, $41$ will play more roles in this section.\\
As a direct consequence of the definition of $ndh$-numbers we have the following:
\begin{proposition}\label{abcharmnd}
The $abc$-conjecture is true on the set of harmonic numbers joined with a finite set of $ndh$-numbers.
\end{proposition}
\begin{proof}
The only way to possibly get infinitely many solutions is if at most one of the numbers is an $ndh$-number.
\end{proof}

\vspace{5mm}

The following result shows  in four different ways that there are infinitely many $ndh$-numbers.
\begin{theorem}\label{infinitendh}
The following assertions hold:\\
i) $2^n41$ is an $ndh$-number for all n.\\
ii) $3^n41$ is an $ndh$-number for all n.\\
iii) If $x$ is an $ndh$-number then either $2x$ or $3x$ is an $ndh$-number.\\
iv)  Any prime number of the form $48k+41$ is an $ndh$-number. Note that by Dirichlet's Theorem \cite{a} this set is infinite because $1=(48,41)$. 
\end{theorem}
\begin{proof}
i) If $n\le 1$ this follows from Theorem \ref{ndhn100}. If $n\ge 2$ and $2^n41$ is a difference of harmonic numbers then we must have $2^n41=3^k-1$ so $2^n41=(2+1)^k-1=2^k+k2^{k-1}+\ldots+2k$ and hence $k$ is even. Since neither $4\cdot 41+1$ nor $8\cdot 41+1$ are powers of $3$ it follows that $n\ge 4$. But then $k=2l$ and $2^n41=3^{2l}-1=(9-1)(9^{l-1}+9^{l-2}+\ldots+9+1)$, so we get that $l$ is even. It follows that $3^{2l}-1$ ends in $0$ so $5\mid 2^n41$, a contradiction.\\ 
ii) By Theorem \ref{ndhn100} we assume that $n\ge 2$. Since $3^n41$ is odd, if it is a difference of harmonic numbers we need to have (after possibly canceling the 3's) that $3^m41=2^k-1$ where $k\ge 3$. Now the left hand side is congruent to 1 or 3 modulo 8 while the right hand side is congruent to 7 modulo 8, a contradiction.\\
iii) If  $x$ is not divisible by $2$ or $3$ this is easy, because if both $3x$ and $2x$ are differences of harmonic numbers we have that $3x=2^m-1$ and $2x=3^n-1$. Subtracting the two equalities we get that $x$ is a difference of harmonic numbers, a contradiction.\\
The general case is hard. Let $x=2^{a-1}3^{b-1}y$, where $2\nmid y$ and $3\nmid y$ and assume that $2x=2^z3^w-2^s3^t$ and $3x=2^u3^v-2^k3^r$. Then $z, s\ge 1$ would contradict the fact that $x$ is an $ndh$-number, and if just one of them is at least 1 we get that 2 divides a power of 3. In conclusion, we get that $2x=3^w-3^t$, and by the Fundamental Theorem of Arithmetic we obtain that $t=b-1$, so $2^ay=3^c-1$, where $c=v-b+1$. Similarly we get that $3^by=2^d-1$. Then we can write $y$ in two ways:
$$\frac{3^c-1}{2^a}=\frac{2^d-1}{3^b}.$$
This means that $3^{b+c}-2^{a+d}=3^b-2^a$. By \cite[Theorem 4]{sc} or \cite{stt} (the proof is too long to include) there are only three solutions. The first one is $a=b=c=1$ and $d=2$ which gives $y=1$. The second one is $a=3$, $b=1$, and $c=d=2$ which also gives $y=1$. Finally, the third one is $a=4$, $b=1$, and $c=d=4$ which gives $y=5$. Since neither 1 nor 5 are $ndh$-numbers, the proof is complete.\\
iv) Assume that $p=48k+41$ is prime. Because $p$ is odd, if $p$ is a difference of harmonic numbers we are in one of the following three cases.\\
{\it Case 1.} $p=48k+41=3^n-2^m$ with $m\ge 1$. Taking remainders modulo 8 on both sides we see that $m\ge 3$ and $n=2t$ is even. If $m$ is odd then $48k+41=3^n-3^m+m3^{m-1}-\ldots-3m+1$ so $3\mid 40$, a contradiction. Hence $m=2s$ is also even. Now $p=48k+41=(3^t-2^s)(3^t+2^s)$ and since $p$ is prime we get $3^t-2^s=1$ so by Theorem \ref{gersonides} we get $t=s=1$ or $t=2$ and $s=3$. This means $n=m=2$ or $n=4$ and $m=6$ none of which are possible.\\
{\it Case 2.} $p=48k+41=2^m-3^n$ so $m\ge 6$. The remainders modulo 8 are 1 on the left and 7 or 5 on the right, a contradiction.\\
{\it Case 3.} $p=48k+41=2^s3^t-1$. Then $48k+42=6(8k+7)=2^s3^t$. By the Fundamental Theorem of Arithmetic we get $s=1$ and $8k+7=3^{t-1}$. The remainders modulo 8 are 7 on the left and 1 or 3 on the right, a contradiction.
\end{proof}

We remark that it is not true that if $x$ is an $ndh$-number then $2x$ is an $ndh$-number. Since $91$ is an $ndh$-number by Theorem \ref{ndhn100}, if this would be true then $2^391$ would be an $ndh$-number. However $2^391=728=3^6-1$. The implication $x$ is an $ndh$-number implies $3x$ is an $ndh$-number fails as well. We have that $85$ is an $ndh$-number by Theorem \ref{ndhn100}, but $3\cdot 85=255=2^8-1$.\\[1mm]
%We have that $1365=3\cdot 5\cdot 91$ is an $ndh$-number. Indeed, since $1365$ is an odd multiple of $3$ we could only have either $1365=3\cdot 2^m-3$ which cannot happen because $5\cdot 91+1=456$ is not a power of $2$ or $1365=2^m-1$ which cannot happen because $1366$ is also not a power of $2$. But now $3\cdot 1365=4095=2^{12}-1.$ \\[1mm]

We can now add infinitely many numbers to the set on which the $abc$-conjecture holds:
\begin{corollary}\label{abcharmprimes}
The $abc$-conjecture is true on the set of harmonic numbers joined with the infinite set of primes of the form $48k+41$.
\end{corollary}
\begin{proof}
Clearly $a$, $b$, and $c$ cannot be all prime. If two of them are prime and one of them is $c$ then $rad(abc)>c$. We show now that we can't have that $a$ and $b$ are prime and $c$ is harmonic. Indeed, if this is the case we get $48K+82=2^s3^t$ and $t=0$ because $3\nmid 82$. Then we get $24k+41=2^{s-1}$ which is a contradiction because the left hand side is odd.\\
Finally, the case when two of the numbers are harmonic: if the prime is $c$ then the radical is big. If the prime is $a$ or $b$ there are no solutions by Theorem \ref{infinitendh} iv).
\end{proof}

We end this section by investigating in how many ways the Fermat primes can be written as a difference of harmonic numbers. Recall that a Fermat prime is a prime number of the form $F_k=2^{2^k}+1$. So far only five Fermat primes are known: $F_0=3, F_1=5, F_2=17, F_3=257,$ and $F_4=65537$. In the next result we investigate how a Fermat prime can be written as a difference of harmonic numbers.
\begin{theorem}\label{fermatdiff} The following assertions hold:\\
i) The only ways to write $3$ as a difference of harmonic numbers are:\\ $4-1, 6-3, 9-6, 12-9,$ and $27-24$.\\
ii) The only ways to write $5$ as a difference of harmonic numbers are:\\ $6-1$, $9-4$, $8-3$, and $32-27$.\\
iii) The only ways to write $17$ as a difference of harmonic numbers are:\\ $18-1$ and $81-64$.\\
iv) Any Fermat prime $F_k=2^{2^k}+1$ with $k\ge 3$ is an $ndh$-number.
\end{theorem}
\begin{proof}
i) Let $3=h-k$, where $h,k$ are harmonic numbers. If none of $h$ and $k$ are divisible by 3 then, since one of them is odd and the other one is even it follows that $k=1$ and $h$ is a power of two, so $h=4$ and we obtain  the first difference. If both $h$ and $k$ are divisible by 3, then $h=3h_1$ and $k=3k_1$, where $h_1$ and $k_1$ are consecutive harmonic numbers and so by Theorem \ref{gersonides} we obtain the last four differences.\\
ii) The first two cases are really easy:  $5=2^s3^t-1$ gives us the first difference: $5=6-1$. The second case $5=3^n-2^m$ is Problem 9 in Section XVI of \cite{gp2} and is also very easy: assume that $5=3^n-2^m$ and note that $m\ge 2$. On the other hand $m$ cannot be $\ge 3$ because the remainders modulo 8 on the two sides would not match (5 on the left and 1 or 3 on the right) so $m=n=2$ and this gives us the second difference in the statement: $5=9-4$. As Lenstra says, sometimes all the difficulty hides in the last case: we have to solve $5=2^m-3^n$. This is a lot tougher than it looks. For the sake of completeness we will give the ingenious proof of Guy, Lacampagne, and Selfridge from \cite{gls}, as presented in \cite{st}. We will denote by $U_n$ the group of units of $\mathbb{Z}_n$. Buckle up, here we go: we first find the last two differences by inspection and show there are no other solutions. We write $5=2^m-3^n=2^5-3^3$. Then $2^5(2^a-1)=3^3(3^b-1)$ where $a=m-5$ and $b=n-3$. We assume that $a\ge 1$ and $b\ge 1$ and look for a contradiction.  Now $27=3^3\mid 2^a-1$ but $81\nmid 2^a-1$ so $9\mid a$ (because $18=ord(2)_{U_{27}}$) but $27\nmid a$ (because $54=ord(2)_{U_{81}}$). Now $2^5=32\mid 3^b-1$, so $8=ord(3)_{U_{32}}\mid b$. Then, using the factorization tables in \cite{blstw} we find our friend 41 playing a role here as well: $41\mid 3^8-1=41\cdot 160$, so $41\mid 3^b-1$ hence $41\mid 2^a-1$ and therefore $20=ord(2)_{U_{41}}\mid a$. Now $11\mid 2^{20}-1=11\cdot 95325$ so $11\mid 2^a-1$. Hence $11\mid 3^b-1$ so $5=ord(3)_{U_{11}}\mid b$. Since $7=2^3-1\mid 2^a-1$ we obtain that $7\mid 3^b-1$ so $6=ord(3)_{U_7}\mid b$. It follows that $5\cdot 6=30\mid b$ and since $271\mid 3^{30}-1=271\cdot 759745874888$ so $271\mid 2^a-1$ and $27\cdot 5=135=ord(2)_{U_{271}}\mid a$, a contradiction.\\
We will prove iii) and iv) together. Let $k\ge 2$ and try to write $F_k=2^{2^k}+1$ as a difference of harmonic numbers. We have the following possibilities:\\
{\it Case 1.} $2^{2^k}+1=2^s3^t-1$. Then $2(2^{2^k-1}+1)=2^s3^t$, so by the Fundamental Theorem of Arithmetic $s=1$ and $2^{2^k-1}+1=3^t$. By Theorem \ref{gersonides} we get that $k=t=2$ (recall that $k\ge 2$). In conclusion we get the first difference in iii): $17=18-1$.\\
{\it Case 2.} $2^{2^k}+1=3^n-2^m$. It is easy to see that $m\notin\{0,1,2\}$, so after taking remainders modulo 8 on both sides we see that $n=2r$ is even. Now if $m$ is odd we get $2^{2^k}+1=3^n-(3-1)^m=3^n-3^m+m3^{m-1}-\ldots-3m+1$, so $3\mid 2^{2^k}$, a contradiction. Therefore $m=2t$ is also even. Now $2^{2^k}+1=(3^r-2^t)(3^r+2^t)$ so $3^r-2^t=1$ and by Theorem \ref{gersonides} $r=t=1$ or $r=2$ and $t=3$. The first option is not possible, so we are left with $n=4$ and $m=6$ which gives us the second difference in iii): $17=81-64$.\\
{\it Case 3.} $2^{2^k}+1=2^m-3^n$. Since $m\ge 3$ this cannot happen because reminders modulo 8 on the two sides do not match (1 on the left and 7 or 5 on the right).\\
This concludes the proof of the theorem because in all cases with solutions we ended up with $k=2$.
\end{proof}

We now add all Fermat primes to the set of harmonic numbers and we prove that the $abc$-conjecture still holds on this new expanded set.

\begin{corollary}
The $abc$-conjecture is true on the set of harmonic numbers joined with the set of Fermat primes.
\end{corollary}
\begin{proof}
A solution cannot have all three primes because they are all odd. If $c$ is a prime, then $rad(abc)>c$. So we have to look at the case when one or both of $a$ and $b$ are primes. The case when only one of them is prime is covered by Theorem \ref{fermatdiff}. Now if $2^{2^k}+1+2^{2^l}+1=2^s3^t$ then $2(2^{2^k-1}+2^{2^l-1}+1)=2^s3^t$. By the Fundamental Theorem of Arithmetic we get that $s=1$ and $2^{2^k-1}+2^{2^l-1}+1=3^t$. Since both exponents on the left are odd it follows that the remainder modulo 3 on the left is $2+2+1\equiv 2$ (mod $3$), a contradiction.
\end{proof}

\section{Numbers that can be written as differences of harmonic numbers}
A Mersenne prime is a prime number of the form $2^p-1$ (it is easy to see that if $2^p-1$ is prime, then $p$ is also prime). The first three Mersenne primes are 3, 7, and 31, corresponding to values of $p$ 2, 3, and 5.  There are currently less than 50 known Mersenne primes. In this section we investigate how a Mersenne prime can be written as a difference of harmonic numbers.
\begin{theorem}\label{mersennediff} The following assertions hold:\\
i) The only ways to write $3$ as a difference of harmonic numbers are:\\ $4-1, 6-3, 9-6, 12-9,$ and $27-24$.\\
ii) The only ways to write $7$ as a difference of harmonic numbers are:\\ $8-1, 9-2,$ and $16-9$.\\
iii) For any Mersenne prime $2^p-1$ with $p\ge 5$ there is no other way to write it as a difference of harmonic numbers.
\end{theorem}
\begin{proof}
i) This was proved in Theorem \ref{fermatdiff} i) (3 is also a Fermat prime).\\
ii) This statement is actually the union of Problems 1 and 10 in Section XVI of \cite{gp2}, but again we will give a proof for the sake of completeness (our proof is essentially the same as the one given in \cite{gp2}).\\
If $3^n=2^m+7$, then $n\ge 2$. We cannot have $m\ge 3$ because the remainder of $3^n$ modulo 8 cannot be 7 (it is either 1 or 3). Therefore $m\in\{0,1,2\}$ and the only solution is $m=1$ and $n=2$, which gives the difference $9-2$.\\
If $2^m-3^n=7$, then $m\ge 3$. If $n=0$, then $m=3$ and this gives us the difference $8-1$. If $n\ne 0$, then the remainder of $2^m=3^n+7=3^n+6+1$ modulo 3 is 1, so $m=2l$ is even and $l\ge 2$, because $m\ge 3$. On the other hand, $3^n=2^m-7=2^m-8+1$ has remainder 1 modulo 8, so $n=2k$ is also even. Then $7=2^m-3^n=2^{2l}-3^{2k}=(2^l-3^k)(2^l+3^k)$. Therefore $2^l-3^k=1$, and so $l=2$ and $k=1$ by Theorem \ref{gersonides}. This gives us the last difference, $16-9$.\\
iii) Let $p\ge 5$ be a prime, and assume that $2^p-1=2^m-3^n$. It follows that $m\ge 5$. If $n=0$, then $m=p$. We assume that $n\ne 0$ and look for a contradiction. We have that $3^n=2^m-2^p+1\equiv 1$ (mod 8), so $n=2k$ is even. If $m$ is odd, we have $2^p-1=(3-1)^m-3^n=3^m-m3^{m-1}+\ldots+3m-1-3^n$, and so $3\mid 2^p$, which is not possible. Thus $m=2l$ is also even. Now $2^p-1=2^{2l}-3^{2k}=(2^l-3^k)(2^l+3^k)$. Since $2^p-1$ is prime, we get that $2^l-3^k=1$, so by Theorem \ref{gersonides} $l=2$, therefore $m=4$, which contradicts $m\ge 5$.\\
We assume now that $2^p-1=3^n-2^m$, so $n\ge 4$. Since $2^p-1\equiv 7$ (mod 8) and $3^n\equiv$ 1 or 3 (mod 8), it follows that $m\le 2$. The only possibility is $3^n\equiv$ 3 (mod 8) and $m=2$. Then $2^p-1=3^n-4$, from which we get again that $3\mid 2^p$, which is a contradiction and the proof is complete.
\end{proof}
Theorem \ref{mersennediff} allows us to obtain our last expansion of the set on which the $abc$-conjecture holds by adding the Mersenne primes to the set of harmonic numbers. 
%Compared to Corollary \ref{abcharmnd} this result seems a little better, because we might be actually adding an infinite set (it is not known whether the set of Mersenne primes is finite).
\begin{corollary}
The $abc$-conjecture is true on the set of harmonic numbers joined with the set of Mersenne primes.
\end{corollary}
\begin{proof}
Let's see how many of the three numbers in the statement of the conjecture can be Mersenne primes. It is clear that not all of them can be Mersenne primes, because the sum of two Mersenne primes is even and thus can't be a Mersenne prime.\\
Let's see if two of them can be Mersenne primes (and so the third one must be harmonic).We start by showing that the sum of two Mersenne primes cannot be harmonic, with the exception of $3+3=6$. Indeed, if $2^p-1+2^q-1=2^r3^s$ we have that $$2(2^{p-1}+2^{q-1}-1)=2^r3^s,$$ so $r=1$ by the Fundamental Theorem of Arithmetic, and 
\begin{equation}\label{eq1}
2^{p-1}+2^{q-1}-1=3^s.
\end{equation}
If both $p$ and $q$ are greater than 4, then the left hand side of (\ref{eq1}) is congruent to 7 modulo 8, while the right hand side is congruent to 1 or 3 modulo 8. Therefore one of them, say $p$, has to be 2 or 3. If $p=2$, it follows that $2^{q-1}+1=3^s$, so by Theorem \ref{gersonides} we get $q-1=1$ or $q-1=3$. The first case gives the solution $3+3=6$, while the second one is not acceptable because $4$ is not prime. If $p=3$ we get that $2^{q-1}+3=3^s$ which is another contradiction.\\
%Let's see now if the difference of two Mersenne primes can be harmonic. If $2^p-1-2^q+1=2^r3^s$, with $q<p$, then $2^p-2^q=2^r3^s$, or $2^q(2^{p-q}-1)=2^r3^s$. Again by the Fundamental Theorem of Arithmetic we get that $q=r$ and $2^{p-q}-1=3^s$. By Theorem \ref{gersonides} we get that $p-q=2$, and $s=1$. Therefore $P=2^p-1$ and $Q=2^q-1$ are twin Mersenne primes. Since it is not known whether there are infinitely many twin Mersenne primes it looks like we might be in trouble. Luckily, in this case we have that $c=P$ and $\{a,b\}=\{Q,2^q3\}$, and so $c=P<6PQ=rad(abc)$ and our corollary is saved.\\
Now, if $c$ is a Mersenne prime, then it is smaller than $rad(abc)$. Finally, the case when one of $a$ or $b$ is the only Mersenne prime in the triple is covered by Theorem \ref{mersennediff} (we note that the only solution in this latter case, $2^p-1+1=2^p$ is also a low quality solution, because $2^p<rad(2^p(2^p-1))=2(2^p-1)$).
\end{proof}

\section*{Acknowledgments} %%% ADD HERE YOUR THANKS AND ACKNOWLEDGMENTS
 
We thank Wai Yan Pong for useful conversations and Crosby Lanham for helping with Java. We also thank Alexandru Gica, Constantin Manoil, Frank Miles, and Rob Niemeyer, who read the manuscript, made valuable suggestions, and corrected errors, and Paltin Ionescu for his uncanny ability (and speed) to find typos.

%%%%%%%% AUTHORS' INFORMATION. DELETE/ADD AUTHORS AS NEEDED
{\footnotesize  
\medskip
\medskip
\vspace*{1mm} 
 
\noindent {\it Natalia da Silva}\\  
California State University, Dominguez Hills \\
1000 E Victoria St \\
Carson, CA 90747\\
E-mail: {\tt ndasilva1@toromail.csudh.edu}\\ \\  

\noindent {\it Serban Raianu}\\  
California State University, Dominguez Hills \\
Mathematics Department\\
1000 E Victoria St \\
Carson, CA 90747\\
E-mail: {\tt sraianu@csudh.edu}\\ \\

\noindent {\it Hector Salgado}\\  
California State University, Dominguez Hills \\
1000 E Victoria St \\
Carson, CA 90747\\
E-mail: {\tt hsalgado1@toromail.csudh.edu} }\\ \\   

\vspace*{1mm}\noindent\footnotesize{\date{ {\bf Received}: January 1, 2017\;\;\;{\bf Accepted}: January 1, 2017}}\\
\vspace*{1mm}\noindent\footnotesize{\date{  {\bf Communicated by Some Editor}}}

\end{document}